\documentclass[12pt, reqno]{amsart}
\usepackage{amsmath, amsthm, amscd, amsfonts, amssymb, graphicx, color}

\newtheorem{df}{Definition}[section]
\newtheorem{thm}[df]{Theorem}
\newtheorem{pro}[df]{Proposition}
\newtheorem{cor}[df]{Corollary}

\begin{document}
\setcounter{page}{1}

\title[Two dimensional Lie algebra of operators]{On The Joint Spectra Of The  Two Dimensional \\  Lie Algebra Of Operators 
In Hilbert Spaces}
\author{Enrico Boasso}
\begin{abstract}We consider the complex solvable non-commutative two dimensional Lie 
algebra $L$, $L=<y>\oplus <x>$, with Lie bracket $[x,y]
=y$, as linear bounded
operators acting on a complex Hilbert space $H$. Under the assumption $R(y)$ closed,
we reduce the computation of the joint spectra $Sp(L,E)$, $\sigma_{\delta ,k}(L,E)$ and $\sigma_{\pi ,k}(L,E)$, 
$k= 0,1,2$, to the computation of the spectrum, the approximate point spectrum,
and the approximate compression spectrum of a single operator. Besides, we also study the case 
$y^2=0$, and we apply our results to the case $H$ finite dimensional.\end{abstract}
\maketitle
\section{Introduction}

In [1] we introduced a joint spectrum for complex solvable finite 
dimensional Lie algebras of operators acting on a Banach space $E$. If $L$ is such
an algebra, and $Sp(L,E)$ denotes its joint spectrum, $Sp(L,E)$ is a compact non empty
subset of $L^*$, which also satisfies the projection property for ideals, i. e., 
if $I$ is an ideal of $L$ and $\Pi\colon L^*\to I^*$ denotes the restriction map, then
$Sp(I,E) = \Pi(Sp(L,E))$. In addition, when $L$ is a commutative algebra, $Sp(L,E)$
reduces to the Taylor joint spectrum, see [5]. Moreover, in [2] we extended
S\l odkowski joint spectra $\sigma_{\delta , k}$ and $\sigma_{\pi , k}$ to the case under consideration and we proved the
usual spectral properties: they are compact non empty subsets of $L^*$ and the projection
property for ideals still holds.\par

In this paper we consider the complex solvable non-commutative  two
dimensional Lie algebra $L$, $L = <y>\oplus <x>$, with Lie bracket $[x,y] = y$,
as bounded linear operators acting on a complex Hilbert space $H$, and we compute
the joint spectra $Sp(L,H)$, $\sigma_{\delta , k}(L,H)$ and $\sigma_{\pi , k}(L,H)$,
for $k= 0, 1, 2$, when $R(y)$ is a closed subspace of $H$. Besides, by means of an homological argument, we reduce the 
computation of these spectra to the one dimensional case. We prove that these
joint spectra are determined by the spectrum, the approximate point spectrum,
 and the approximate compression spectrum of $x$ in $Ker(y)$ and $\overline x$
in $H/R(y)$, where $\overline x$ is the quotient map associated to $x$, ($R(y)$ and $Ker(y)$ are invariant subspaces for the operator $x$).\par

In addition, we consider the case $y^2 = 0$ (it easy to see that $y$ is a nilpotent operator),
and we obtain a relation between the spectrum of $x$ in $R(y)$  and a subset of the spectrum
of $\overline x$ in $H/R(y)$, which give us a more precise characterization
of the joint spectrum $Sp(L,E)$. Finally, we apply our computation to the case
 H finite dimensional.\par

The paper is organized as follows. In Section 2 we review several definitions
and results of [1] and [2]. In Section 3 we prove our main theorems and,
in Section 4, we consider the case $y^2 = 0$ and the finite dimensional case.\par

\section{ Preliminaries}

In this section we briefly recall the definitions of the joint spectra
$Sp(L,H)$, $\sigma_{\delta , k}(L,H)$ and $\sigma_{\pi}(L,H)$, $k= 0, 1, 2$. We
restrict ourselves to the case under consideration. For a complete account of
the definitions and mean properties of these joint spectra, see [1] and [2].\par
From now on, let $L$ be the complex solvable two dimensional Lie algebra,
$L = <y>\oplus <x>$, with Lie bracket $[x,y] = y$, which acts as right continuous
linear operators on a Hilbert space $H$, i. e., $L$ is a Lie subalgebra of 
$\mathcal{L} (H)^{op}$, where $\mathcal{L} (H)$ is the algebra of all bounded linear
operators defined on $H$, and where $\mathcal{L} (H)^{op}$ means that we consider
$\mathcal{L} (H)$ with its opposite product. We observe that any complex solvable non-commutative two
dimensional Lie algebra may be presented in the above form.\par
If $f$ is a character of $L$, we consider the chain complex 
$(H\otimes\wedge L, d(f))$, where $\wedge L$ denotes the exterior algebra
of $L$, and $d(f)$ is the following map:
$$
d_{p-1}(f)\colon H\otimes\wedge^p L\to H\otimes\wedge^{p-1} L,
$$
$$
d_0(f)(a<y>) = y(a),\hskip1.5cm d_0(f)(b<x>) = (x-f(x))(b),
$$
$$
d_1(f)(c<yx>) = (-(x-1-f(x)))(c)<y> + y(c) <x>.
$$
Let $H_*(H\otimes\wedge L,d(f))$ denote the homology of the complex
$(H\otimes\wedge L,d(f))$, we now state our first definition.\par
\begin{df}With $H$, $L$ and $f$ as above, the set $\{f\in L^*\colon f(L^2) = 0,
H_*(H\otimes\wedge L,d(f))\ne 0\}$, is the joint spectrum of $L$ acting on
$H$, and it is denoted by $Sp(L,H)$.\end{df}
As a consequence of the results of [1], we have that $Sp(L,H)$ is
a compact non empty subset of $L^*$. Besides, as a standard calculation shows that
the equality $y = [x,y]^{op} = [y,x]$ implies $ny^n = [y^n,x] = [x,y^n]^{op}$, we have that $y$
is a nilpotent operator. Thus, $Sp(<y>) = 0$, and by the projection property, if 
$f$ belongs to $Sp(L,H)$, as $<y> = L^2$ is an ideal of $L$, $f(y) = 0$.\par
Now, let us consider the basis of $L$, $A$, defined by, $A =\{ y,x\}$, and $B$, the basis
of $L^*$ dual of $A$. If we consider $Sp(L,H)$ in terms of the above basis,
and we denote it by $Sp((y,x),H)$, i. e., $Sp((y,x),H) = \{ (f(y),f(x))\colon f\in Sp(L,H)\}$, we have that,
$Sp((y,x),H) = \{(0,f(x))\colon f\in Sp(L,H)\}$.\par
In addition, the complex $(H\otimes\wedge L,d(f))$ may be written in 
the following way,
$$
0\rightarrow H \xrightarrow{d_1} H\oplus H \xrightarrow{d_0} H\rightarrow 0,
$$
$$
d_0 = \begin{pmatrix}
                y& x-\lambda\\ 
                             \end{pmatrix},
\hskip2cm 
d_1 = \begin{pmatrix}
                     -(x-1-\lambda )\\
                      y\\ \end{pmatrix} ,
$$          
where $\lambda = f(x)$. We denote this chain complex by $(C,d(\lambda))$. 
Thus, as $(0,\lambda)\in Sp((y,x),H)$ if and only if $f\in Sp(L,H)$, 
where $\lambda = f(x)$, to compute the latter is equivalent to compute the former,
and to study the exactness of the chain complex $(H\otimes\wedge L,d(f))$
is equivalent to study the exactness of $(C,d(\lambda))$.\par
With regard to the joint spectra $\sigma_{\delta ,k} (L,H)$ and
$\sigma_{\pi , k}(L,H) $, $k=0,1,2$, we review, for the case under consideration, the definition of them given in [2]. If $p =0, 1, 2$, let $\Sigma_p (L,H)$
be the set $\Sigma_p (L,H) = \{ f\in L^*\colon f(L^2)= 0,  H_p((H\otimes\wedge L,d(f)))\ne 0\}$.
We now state our second definition.\par
\begin{df} With $H$, $L$ and $f$ as above, 
$$
\sigma_{\delta ,k} (L,H) =\bigcup_{0\le p\le k} \Sigma(L,H),
$$
$$
\sigma_{\pi ,k}(L,H)= \bigcup_{k\le p\le 2}\Sigma_p (L,H)\bigcup \{f\in L^*\colon f(L^2) =0, R(d_k(f)) \hbox{is not closed}\},  
$$
where $0\le k\le 2$.\end{df}
We observe that $Sp(L,H)= \sigma_{\delta , 2}(L,H)=\sigma_{\pi ,0}(L,H)$. Besides, as we have said, these joint spectra are compact non empty subsets of $L^*$.
In addition, as in the case of the joint spectrum $Sp(L,H)$, we consider the
joint spectra $\sigma_{\delta ,k}(L,H)$ and $\sigma_{\pi ,k}(L,H)$ in terms
of the basis $A$ and $B$. As these joint spectra are subsets of $Sp(L,H)$, we have
that $\sigma_{\delta , k}((y,x),H) = \{(0,f(x))\colon f\in \sigma_{\delta ,k}(L,H)\}$,
and $\sigma_{\pi ,k }((y,x),H) = \{ (0,f(x))\colon f\in \sigma_{\pi ,k}(L,H)\}$,
where $k = 0, 1, 2$.\par

Moreover, as in the case of the joint spectrum $Sp(L,H)$, to compute
$\sigma_{\delta , k} (L,H) $ and $\sigma_{\pi ,k}(L,H)$, $0\le k\le 2$, is equivalent to 
compute these joint spectra in terms of the basis $A$ and $B$. Finally, to
compute the latter joint spectra it is enough to study the complex $(C, d(\lambda))$,
and to consider the corresponding properties involved in the definition of
$\sigma_{\delta ,k}(L,H)$ and $\sigma_{\pi , k}(L,H)$, $0\le k\le 2$, for it.\par
\vskip4pt
\section{The Main Result}
We begin with the characterization of $Sp(L,H)$. Indeed, we
consider $Sp((y,x),H)$, and by means of an homological argument we reduce
its computation to the case of a single operator.\par
Let us consider the chain complex $(\overline C, \overline d)$,
$$
0\rightarrow H\xrightarrow{\overline d = y} H\rightarrow 0.
$$
Then an easy calculation shows that we have a short exact sequence of chain complex of 
the form,
$$
0\rightarrow (\overline C,\overline d)\xrightarrow{i}(C,d(\lambda))\xrightarrow{p}(\overline C,\overline d)\rightarrow 0,
$$
where $(i_j)_{(0\le j\le 2)} $ and $(p_j)_{(0\le j\le 2)}$ are the following maps:
$i_2=0$, $i_1=I_{H}\oplus 0$, $i_0=I_{H}$, and $p_2=I_{H}$, $p_1=0\oplus I_{H}$, 
$p_0= 0$.\par
Thus, by [4, Chapter II, Section 4, Theorem 4.1], and the fact that $p$ is a map of degree $-1$, we have a 
long exact sequence of homology spaces of the form,
$$
\rightarrow H_q(C,d(\lambda))\xrightarrow{p_{q*}} H_{q-1}(\overline C, \overline d)\xrightarrow{\partial_{q-1}} H_{q-1}(\overline C,\overline d)\xrightarrow{i_{q-1 *}}H_{q-1}(C,d(\lambda))\rightarrow.
$$
We observe that $H_1(\overline C, \overline d)= Ker(y)$ and that
$H_0(\overline C,\overline d)= H/R(y)$. Moreover, as $[x,y]^{op} = y$, we have
that $x(R(y))\subseteq R(y)$ and that $x(Ker(y))\subseteq Ker(y)$. Then, by
[4, Chapter II, Section 4, Theorem 4.1], $\partial_q$, $q=0, 1$, are the following maps: $\partial_0 ([a])=
[(x-\lambda)(a)] = (\overline x-\lambda)[a]$, and $\partial_1(b)= -(x-\lambda -1)(b)$,
where $\overline x\colon H/R(y)\to H/R(y)$ is the map obtained by passing $x$ to the 
quotient space $H/R(y)$. We now give our characterization of $Sp(L,H)$.\par
\begin{pro} Let $L$ be the complex solvable non-commutative two dimensional Lie 
algebra $L=<y>\oplus <x>$, with Lie bracket [x,y]=y, which acts as right continuous
linear operators on a complex Hilbert space $H$. If $R(y)$ is a closed
subspace of $H$ and  we consider $Sp(L,H)$ in terms of the basis $\{y,x\}$ of
$L$ and the basis of $L^*$ dual of the latter, then we have,
$$
Sp((y,x), H)=\{0\}\times Sp(x-1,Ker(y))\cup \{0\}\times Sp(\overline x,H/R(y)).
$$
\indent In addition, we have:\par
\noindent {\rm (i)} $H_0(C,d(\lambda))=0$ iff $\overline x-\lambda\colon H/R(y)\to H/R(y)$ is a surjective map,\par
\noindent  {\rm (ii)} $H_2(C,d(\lambda))=0$ iff $x-1-\lambda\colon ker(y)\to Ker(y)$ is an injective map,\par
\noindent {\rm (iii)} $H_1(C,d(\lambda))=0$ iff $\overline x-1-\lambda$ is injective, and 
$x-\lambda-1$ is surjective.\end{pro}
\begin{proof}
It is a consequence of the long exact sequence of homology spaces and 
the form of the maps $\partial_j$, $j=0, 1$.
\end{proof}
In order to characterize the joint spectra $\sigma_{\pi , k}(L,H)$, we recall
the notion of approximate point spectrum of an operator $T$: $\lambda$ is in
the approximate point spectrum of $T$, which we denote by $\Pi (T)$, if there exists a sequence
of unit vectors, $(x_n)_{n\in \mathbb N}$, $x_n\in H$, $\parallel x_n\parallel=1$, such that $(T-\lambda)(x_n)\to 0$ ($n\to \infty$).
An easy calculation shows that $\lambda\notin\Pi (T)$ if and only if 
$Ker (T-\lambda) = 0$ and $R(T-\lambda)$ is closed in $H$. \par
We now consider the spectrum $\sigma_{\pi , 2}((y,x),H)$. We observe that,
as $[x,y]^{op} = y$, $(x-1)(Ker(y)\subseteq Ker(y)$. Then, we may consider 
$\Pi(x-1, Ker(y))$. Indeed, we shall see
that $\sigma_{\pi ,2}((y,x),H) = \{0\}\times \Pi (x-1, Ker(y))$.\par
To prove the last assertion we proceed as follows. By Definition $2$,
we have that $\sigma_{\pi ,2 }^c = \{(0,\lambda)\colon  H_2(C, d(\lambda))=0 \hbox{ and } R(d_1(\lambda)) \hbox{ is closed}\}$.
However, by the definition of $d_1(\lambda)$ and $H_2(C,d(\lambda))$, $H_2(C,d(\lambda))=
Ker(x-1-\lambda)\cap Ker(y)$. Then, $H_2(C,d(\lambda))=0$ is equivalent to
$Ker(x-1-\lambda\mid_{Ker(y))} =0$. Thus, in order to conclude with our assertion, it is
enough to see that the fact $R(x-1-\lambda\mid_{Ker(y)})$ is closed, is equivalent to
$R(d_1(\lambda))$ is closed. \par
Indeed, if $(a_n)_{n\in \mathbb N}$ is a sequence in $Ker(y)$ such that
$(x-1-\lambda)(a_n)\to b\in Ker(y)$ ($n\to \infty$), we have that, $d_1(\lambda)(a_n)\to (-b,0)$ ($n\to \infty$).
If $R(d_1(\lambda))$ is closed, there is a $z$ in $H$ such that $d_1(\lambda)(z) = (-b,0)$,
i.e., $-(x-1-\lambda)(z)= -b$ and $y(z)=0$. Thus, $z\in Ker(y)$ and $R((x-1-\lambda)\mid_{Ker(y)})$ is closed.\par
On the other hand, if $R((x-1-\lambda\mid_{Ker(y)})$ is closed, let us consider
a sequence $(z_n)_{n\in \mathbb N}$, $z_n\in H$, such that $d_1(\lambda)(z_n)\to (w_1,w_2)\in H\oplus H$ ($n\to\infty$).
We decompose $H$ as the orthogonal direct sum of $Ker(y)$ and $Ker(y)^{\perp}$,
$ H = Ker (Y)\oplus Ker(y)^{\perp}$. Let $(a_n)_{n\in \mathbb N}$ and $(b_n)_{n\in \mathbb N}$
be sequences in $Ker(y)$ and $Ker(y)^{\perp}$, respectively, such that
$z_n= a_n+ b_n$. Then,
\begin{align*}
            d_1(\lambda) &=d_1(\lambda)(a_n) +d_1(\lambda)(b_n)\\
                         &=(-(x-1-\lambda)(a_n),0) +(-(x-1-\lambda)(b_n),\overline y (b_n)),\\ 
                                                                                         \end{align*}
where $\overline y\colon Ker(y)^{\perp}\to R(y)$ is the restriction of $y$ to $Ker(y)^{\perp}$. 
We observe that, as $R(y)$ is a closed subspace of $H$, $\overline y$
is a topological homeomorphism. Besides, as $\overline y (b_n)\to w_2$ ($n\to \infty$),
there exists  a $z_2\in Ker(y)^{\perp}$ such that $b_n \to z_2$ ($n\to \infty$), and
$\overline y (z_2) = w_2$. Then, $-(x-1-\lambda) (b_n)\to -(x-1-\lambda)(z_2)$ ($n\to \infty$),
and $-(x-1-\lambda)(a_n)\to w_1 + (x-1-\lambda)(z_2)$ ($n\to \infty$). As $(a_n)_{n\in \mathbb N}$
is a sequence in $Ker(y)$, and $R(x-1-\lambda\mid_{Ker(y)})$ is closed, there is
a $z_1\in Ker(y)$ such that $w_1 + (x-1-\lambda)(z_2) = -(x-1-\lambda)(z_1)$. Thus,
$(w_1,w_2) = d_1(\lambda)(z_1+z_2)$, equivalently, $R(d_1(\lambda))$ is a closed
subspace of $H\oplus H$.\par
With regard to $\sigma_{\pi , 1}((y,x),H)$, we have, by Definition 2.2, that,
$$
\sigma_{\pi ,1}((y,x),H)^c=\{ (0,\lambda)\colon H_i(C,d(\lambda))=0, i=1,2, \hbox{and } R(d_0(\lambda)) \hbox{ is closed}\} ,   
$$
which, by Proposition 1, is equivalent to the following conditions:\par
\noindent (i) $x-1-\lambda\colon Ker(y)\to Ker(y)$ is an isomorphic map,\par
\noindent (ii) $\overline x-\lambda\colon H/R(y)\to H/R(y)$ is an injective map,\par
\noindent (iii) $R(d_0(\lambda))$ is closed.\par
We shall see that $\sigma_{\pi ,1} ((y,x),H)= Sp(x-1,Ker(y))\cup\Pi (\overline x, H/R(y))$.\par
Indeed, it is clear that condition (i) is equivalent to $\lambda\notin Sp(x-1,Ker(y))$.
Then, it is enough to see that condition (ii)-(iii) are equivalent to $\lambda\notin\Pi (\overline x,H/R(y)$
$)$. However,
by (ii), it suffices to verify that the fact 
$R(d_0)(\lambda)$ is closed is equivalent to $R(\overline x-\lambda)$ is closed.
Now, as the quotient map, $\Pi\colon H\to H/R(y)$, is an identification, by [3, Chapter II, Section 6, Lemma 6.1],
$R =R(\overline x-\lambda)= \Pi (R(x-\lambda))$ is closed in $H/R(y)$ if and only if $\Pi^{-1}(R) =
R(x-\lambda) + R(y) =R(d_0(\lambda))$ is closed in $H$.\par
In order to study the joint spectra $\sigma_{\delta ,k}(L,H)$, $k= 0,1,2$,
we recall the definition of the approximate compression Spectrum of 
an operator $T$ in $H$: $\lambda$ is in the approximate compression spectrum of $T$, which we denote by $\Pi C(T)$,  
if there exists a sequence of unit vectors in $H$, $(x_n)_{n\in \mathbb N}$, $x_n\in H$, $\parallel x_n\parallel= 1$, such
that $(T-\lambda)^*(x_n)\to 0$ ($n\to \infty$), i. e., $\Pi C (T)=\Pi (T^*)$. Besides,
an easy calculation shows that $\lambda$ does not belong to $\Pi (T)$ if and
only if $(T-\lambda)$ is a surjective map.\par
We now consider the joint spectra  $\sigma_{\delta ,o}((y,x),H)$. However, by Definiton 2.2, Proposition 3.1 and the previous considerations  about the approximate compression spectrum, it is
clear that $\sigma_{\delta ,k}((y,x),H) = \{0\}\times\Pi C (\overline x ,H/R(y))$.\par
With regards to $\sigma_{\delta ,1}((y,x),H)$, by Definition 2.2 and Proposition 3.1,
we have that $(0,\lambda)$ does not belong to $\sigma_{\delta , 1}((y,x),H)$ if
and only if $(0,\lambda)$ satisfies the following conditions:\par
\noindent (i) $\overline x-\lambda\colon H/R(y)\to H/R(y)$ is an isomorphic map,\par
\noindent (ii) $x-1-\lambda\colon Ker(y)\to Ker(y)$ is surjective.\par
Then, it is obvious that $\sigma_{\delta , 1} ((y,x), H)= \{0\}\times Sp(\overline x, H/R(y))\cup \{0\}\times \Pi C (x-1\mid_{Ker(y)})$.\par
We now summarize our results.\par
\begin{thm}Let $L$ be the complex solvable
non-commutative two dimensional Lie algebra, $L=<y>\oplus<x>$, with Lie bracket $[x,y]^{op}=y$, which acts as right continuous linear 
operators on a complex Hilbert space $H$. If $R(y)$ is closed, then the joint spectra
$Sp(L,H)$, $\sigma_{\delta ,k}(L,H)$ and $\sigma_{\pi ,k}(L,H)$, $k= 0, 1, 2$,
in terms of the basis $\{y, x\}$ of $L$ and the basis of $L^*$ dual of the latter,
may be characterize as follows:\par
\noindent {\rm (i)} $Sp((y,x),H)=\{0\}\times Sp(x-1, Ker(y))\cup \{0\}\times Sp(\overline x, H/R(y))$,\par
\noindent {\rm (ii)} $\sigma_{\delta , 0}((y,x),H) = \{0\}\times \Pi C (\overline x, H/R(y))$,\par
\noindent {\rm (iii)} $\sigma_{\delta , 1}((y,x))= \{0\}\times Sp(\overline x, H/R(y))\cup\{0\}\times \Pi C(x-1,Ker(y))$,\par
\noindent {\rm (iv)}  $\sigma_{\pi ,2}((y,x),H)= \{0\}\times\Pi (x-1, Ker(y))$,\par
\noindent {\rm (v)}   $\sigma_{\pi ,1}((y,x),H) =\{0\}\times Sp(x-1, Ker(y))\cup \{0\}\times\Pi (\overline x, H/R(y))$,\par
\noindent {\rm (vi)} $\sigma_{\delta ,2} ((y,x), H)=\sigma_{\pi ,0}((y,x),H)=Sp((y,x),H)$.
\end{thm}
\section{A Special Case}
As we have seen, $y$ is a nilpotent operator. In this section we study the case
$y^2=0$ and we obtain a more precise characterization of th joint spectrum $Sp(L,H)$.\par
We decompose $H$ in the following way: $H=Ker(y)\oplus Ker(y)^{\perp} $. Besides,
as $R(y)$ is contained in $Ker(y)$, let us consider $M$, the 
closed subspace of $H$ defined by $M= Ker(y)\cap R(y)^{\perp}$.
Then we have another orthogonal direct sum decomposition of $H$, $H= R(y)\oplus M\oplus Ker(y)^{\perp} $. Moreover, if we
recall that $x(R(y))\subseteq  R(y)$ and $x(Ker(y))\subseteq Ker(y)$, then we have that $x$ and $y$ have the following form,
$$  
y=\begin{pmatrix}  
                0&0&\overline y\\
                0&0&0\\
                0&0&0\\
                                \end{pmatrix},
\hskip2cm
x=\begin{pmatrix}
               x_{11}&x_{12}&x_{13}\\
                0&x_{22}&x_{23}\\
                0&0&x_{33}\\
                                      \end{pmatrix},
$$            
where $\overline y$ is as in Section 3 and the maps $x_{ij}$ ,$1 \le i\le j\le 2$, are
the restriction of $x$ to the corresponding spaces. We now see that, in the case under consideration, $Sp(L,H)$ reduces essentially to the spectrum of $x$ in $Ker(y)$.\par
 \begin{pro} Let $L$ be the complex solvable
non commutative two dimensional Lie algebra $L=<y>\oplus<x>$, with Lie bracket $[x,y]^{op} =y$, which acts as right continuous linear
operators on a complex Hilbert space $H$. If $R(y)$ is closed and $y^2=0$, $Sp(L,H)$, in 
terms of the basis $\{y,x\}$ of $L$ and the basis of $L^*$ dual of the latter,
may be described as follows. If $x_{11}$ and $x_{22}$ are the maps defined above, and if $S_i$, $i=1,2$, are the sets: $S_1=(Sp(x_{11}, R(y))-1)$, and $S_2= (Sp(x_{22},R(y)^{\perp}\cap Ker(y)) $, then, we have that,
$$
Sp((y,x),H)= \{0\}\times (S_1\cup (S_1 +2) \cup S_2\cup (S_2-1)).
$$
\end{pro}
\begin{proof}  
An easy calculation shows that the relation $[x,y]^{op}=y$ is equivalent to
$\overline y x_{33} -x_{11}\overline y =\overline y$. However, as $\overline y$ is
a topological homeomorphism, $x_{33} = I_{Ker(y)^{\perp}} + \overline y^{-1} x_{11} \overline y$. In particular,
$Sp(x_{33},Ker(y)^{\perp})= Sp(x_{11}, R(y)) +1$. Then, as $Sp(\overline x, H/R(y))=
Sp(x_{22},M) \cup Sp(x_{33}, Ker(y)^{\perp})$, where $M = R(y)^{\perp}\cap Ker(y), $we have that
$Sp(\overline x,H/R(y))= (S_1+2)\cup S_2$.\par
On the other hand, it is clear that $Sp(x-1,Ker(y)) = S_1\cup (S_2-1)$. Thus, by Theorem 1,
we conclude the proof.
\end{proof}
Finally, we consider the case $R(y)$ closed, $y^2=0$, and $H$ finite dimensional.
If $r= \dim (R(y)) $ and $k = \dim (Ker(y))$, let us chose a basis of $Ker(y)$ such that the 
first $r$-vectors of it are a basis of $R(y)$, and in this basis, $x$ has an upper triangular form,
with diagonal entries $\lambda_{ii}$, $1\le i\le k$. Then we have the following corollary.\par
\begin{cor} Let $H$, $L$ and the operator $y$ be as in Proposition 4.1. If $H$ is  finite dimensional and
 we consider a basis of $Ker(y)$ with the above conditions, $Sp(L,H)$, in terms
of the basis of $L$ and $L^*$ considered in Proposition 4.1, is the following set,
$$
Sp((y,x),H) =\{0\}\times \{(\lambda_{ii}-1)_{(1\le i\le k)}\cup (\lambda_{ii})_{(m\le i\le k)}\cup (\lambda_{ii} +1)_{(1\le i\le m)}\}.
$$
\end{cor}

\bibliographystyle{amsplain}

\begin{thebibliography}{99}

\bibitem{  } E. Boasso and A. Larotonda,  A spectral theory for solvable
Lie algebras of operators, Pacific J. Math. 158 (1993),  15-22.

\bibitem{  } E. Boasso, Dual properties and joint spectra  for solvable
Lie algebras of operators,  J. Operator Theory 33 (1995),  105-116.

\bibitem{ }  S.T. Hu, Elements of General Topology, Holden Day, 1964.

\bibitem{ } S. Mac Lane, Homology, Springer-Verlag, 1963.

\bibitem{ } J. F. Taylor, A joint spectrum for several commuting operators, J. Funct. Anal. 6 (1976), 172-191.
\end{thebibliography}

\vskip.5cm
\noindent Enrico Boasso\par
\noindent E-mail address: enrico\_odisseo@yahoo.it
\end{document}